\documentclass[letterpaper, 10 pt, conference]{ieeeconf}
\IEEEoverridecommandlockouts
\usepackage{cite}
\usepackage{url}
\usepackage{amsmath,amssymb,amsfonts}
\usepackage{multirow}

\usepackage{amsthm}

\usepackage{graphicx}
\usepackage{textcomp}
\usepackage{epsfig}

\usepackage{epstopdf}
\usepackage[ruled]{algorithm}
\usepackage{algorithmic}
\usepackage{enumerate}
\usepackage{float}
\usepackage{graphics,graphicx}
\usepackage{subfigure}

\usepackage{array}
\usepackage{mathtools}
\usepackage[justification=centering]{caption}
\allowdisplaybreaks[3]
\usepackage{url}
\DeclareGraphicsExtensions{.png}
\graphicspath{{./Pics/}}
\usepackage{color}
\usepackage{xspace}

\newcommand{\E}{\mathbb{E}}

\usepackage{hyperref}
\hypersetup{
    colorlinks=true,
    linkcolor=blue,
    filecolor=gray,      
    urlcolor=blue,
    citecolor=blue,
}

\def\BibTeX{{\rm B\kern-.05em{\sc i\kern-.025em b}\kern-.08em
		T\kern-.1667em\lower.7ex\hbox{E}\kern-.125emX}}

\newtheorem{lemma}{Lemma}
\newtheorem{theorem}{Theorem}

\newtheorem{assumption}{Assumption}

\usepackage{bbm}
\renewcommand{\theequation}{\arabic{equation}}

\begin{document}
	
	\title{\LARGE \bf A Momentum-based Stochastic Algorithm for Linearly Constrained Nonconvex Optimization}

\author{Chenyang Qiu, %\IEEEmembership{Graduate Student Member, IEEE}
Mihitha Maithripala, and
Zongli Lin%, \IEEEmembership{Fellow, IEEE}
\thanks{This work was supported in part by the US National Science Foundation under the grant number CMMI–2243930.} 
\thanks{%
Chenyang Qiu, Mihitha Maithripala, and Zongli Lin are with the Charles L. Brown Department of Electrical and Computer Engineering, University of Virginia, Charlottesville, VA 22904, USA (e-mail: nzp4an@virginia.edu;   wpg8hm@virginia.edu; zl5y@virginia.edu). {\em (Corresponding author: Zongli Lin.)}}
} 
\maketitle	
 
\begin{abstract}
This paper studies a stochastic algorithm for linearly constrained nonconvex optimization, where the objective function is smooth but only unbiased stochastic gradients with bounded variance are available. We propose a momentum-based augmented Lagrangian method that employs a Polyak-type gradient estimator and requires only one stochastic gradient evaluation per iteration. Under the standard stochastic oracle model and the smoothness condition of the expected objective, we establish a convergence guarantee in terms of the first-order KKT residual of the original constrained problem. In particular, the proposed method computes an $\epsilon$-stationary solution in expectation within 
$O(\epsilon^{-4})$ stochastic gradient evaluations. Numerical experiments further show that the proposed method achieves competitive iteration complexity and improved wall-clock efficiency compared with representative recursive-momentum baselines.
\end{abstract}

% \setlength{\abovedisplayskip}{4pt} \setlength{\belowdisplayskip}{4pt}
% \setlength\abovecaptionskip{-1.0pt}
% \setlength\belowcaptionskip{-1.0pt}

% \begin{IEEEkeywords}
% Nonconvex optimization, constrained optimization, stochastic algorithm.
% \end{IEEEkeywords}

% \begin{table}[h]
%     \centering
%     \begin{tabular}{ c| c}
%     \hline
%         si,kS_{i,k} & SoC state \\
%         Δt\Delta t & sampling time\\
%         Pi,kP_{i,k} & power of battery \\
%         xi,kx_{i,k} & state \\
%         Pi\mathcal{P}_{i} & local feasible set \\
%         P\mathcal{P} & global feasible set \\
%     \hline
%     \end{tabular}
%     \caption{Notation}
% \end{table}
\section{Introduction}\label{sec: Introduction}
Many optimization problems arising in modern data-driven systems involve three features simultaneously: nonconvex objectives, structural constraints, and stochastic first-order information \cite{ghadimi2013stochastic}. The combination of these features appears in a broad range of applications, including constrained learning \cite{gaines2018algorithms}, optimal control \cite{o2013splitting, lin2013design}, and networked optimization \cite{zichong2024mixing,qiu2025non}, where the underlying models are often complex, the constraints cannot be ignored, and the exact gradient computation is either infeasible or expensive. Compared with unconstrained stochastic optimization, the presence of constraints introduces an additional difficulty, since the algorithm must consider the feasibility and optimality at the same time while coping with gradient noise. As a result, a big challenge is to develop methods that remain computationally light at each iteration and yet admit rigorous convergence guarantees under realistic stochastic assumptions.

Among such problems, linearly constrained stochastic nonconvex optimization forms a fundamental and practically relevant class. In this paper, we consider problems of the form
\begin{equation}\label{form: equality}
\min_{x \in \mathbb{R}^d} \; f(x) \quad \text{s.t. } Ax = b,
\end{equation}
where $f:\mathbb{R}^d \to \mathbb{R}$ is defined in the expectation form as $f(x)=\E_{\xi\sim\mathcal D}[f(x;\xi)]$ and is $L$-smooth but possibly nonconvex, $\mathcal D$ is an unknown distribution from which independent and identically distributed (i.i.d.) samples can be drawn. In many large-scale or data-driven scenarios, however, the exact gradient of $f$ is unavailable, and one can only access stochastic gradient information through sampling \cite{cutkosky2019momentum}. We assume access to a sampling mechanism that provides i.i.d. stochastic gradients $g(x;\xi) = \nabla f(x;\xi)$. To ensure the tractability of the stochastic optimization process, we impose the following standard assumption on the gradient variance.
\begin{assumption}\label{ass: variance}
For any $x\in \mathbb{R}^d$, the stochastic gradient $g(x;\xi)$ satisfies
\begin{equation*}\E_{\xi\sim\mathcal D}[g(x;\xi)] = \nabla f(x), \quad \E_{\xi\sim\mathcal D}\|g(x;\xi) - \nabla f(x) \|^2 \leq \sigma^2.\end{equation*}
\end{assumption}

\subsection{Related Works}
% Research on linearly constrained optimization has a long history in the convex setting, where first-order augmented Lagrangian, penalty, and primal-dual methods have been studied systematically, with established iteration-complexity guarantees for affine-constrained convex programs. These developments provide the algorithmic foundation for handling linear constraints, but their analyses rely critically on convexity and therefore do not directly extend to the nonconvex composite problems considered here.

Existing methods for linearly constrained nonconvex optimization reveal a clear tradeoff between theoretical strength and practical efficiency.
Recent deterministic works have developed proximal and inexact augmented-Lagrangian frameworks to stabilize primal-dual updates and establish convergence to first-order stationary solutions \cite{hong2016decomposing, zhang2020global,wang2019global}. However, the deterministic literature generally assumes access to exact gradients and is therefore not directly suitable for fully stochastic large-scale learning problems.

% For linearly constrained nonconvex optimization, recent deterministic works have developed proximal and inexact augmented-Lagrangian frameworks to stabilize primal-dual updates and establish convergence to first-order stationary solutions \cite{hong2016decomposing, zhang2020global,wang2019global}. However, the deterministic literature generally assumes access to exact gradients and is therefore not directly suitable for fully stochastic large-scale learning problems. 
On the stochastic side, although basic stochastic penalty and stochastic augmented Lagrangian methods are computationally straightforward, they often have inferior convergence guarantees for linear equality-constrained problems \cite{jin2022stochastic, lin2020single}.

% On the stochastic side, there exist some methods using basic stochastic penalty or augmented Lagrangian methods, which admit bad complexity guarantees for linear equality-constrained problems \cite{jin2022stochastic, lin2020single}. 
To speed up convergence and reduce gradient variance, several studies have combined variance reduction techniques, such as SVRG \cite{johnson2013accelerating} and SARAH \cite{nguyen2017sarah}, both with augmented Lagrangian methods \cite{zheng2016fast, zeng2024accelerated}. Although variance reduction techniques help these methods to achieve a linear convergence rate \cite{zeng2024accelerated}, they require periodic full-gradient evaluations.

Another line of work employs momentum-based gradient estimators, particularly recursive momentum schemes \cite{cutkosky2019momentum}, to improve stochastic gradient tracking in nonconvex optimization. Such estimators have been incorporated into several stochastic primal-dual and penalty-based methods for constrained problems. In particular, \cite{alacaoglu2024complexity,huang2019faster,lu2026variance} establish an $O(\epsilon^{-3})$ complexity bound by using recursive momentum techniques. A key feature of these approaches is that each iteration typically requires two stochastic gradient evaluations in order to update the recursive gradient estimator.
Moreover, achieving the $O(\epsilon^{-3})$ complexity generally relies on stronger stochastic smoothness assumptions, often referred to as sample-level or mean-squared smoothness, which requires that
$$
 \E_{\xi\sim\mathcal D}\|g(x ; \xi)-g(y ; \xi)\| \leq L\|x-y\|, \quad \forall x, y \in \mathbb{R}^d .
$$
This condition assumes that the stochastic gradients corresponding to individual samples are smooth functions of the decision variable. In contrast, the algorithm proposed in the current paper operates under a weaker and more commonly adopted assumption (see Assumption \ref{ass: smoothness}) that only the expected objective function is smooth.
% Ref. \cite{curtis2024worst} proposes a SQP-type method which achieves $O(\epsilon^{-4})$ However, this method typically requires solving equality-constrained quadratic subproblems or KKT linear systems at each iteration and relies on stronger regularity assumptions on the constraint Jacobian, which may lead to higher per-iteration computational cost.
The convergence guarantee in \cite{huang2025stochastic}, although established under the same assumption, is stated in terms of an $\epsilon$-near stationary solution, which does not directly correspond to the exact first-order KKT residual of problem \eqref{form: equality}. In contrast, our analysis characterizes the convergence rate directly in terms of the KKT residual associated with the original constrained problem.

\subsection{Contributions}
The above discussion shows that, for linearly constrained nonconvex optimization, there remains a gap between practical efficiency and theoretical generality in existing stochastic methods. Methods developed under weak stochastic assumptions often have limited convergence performance, while faster approaches typically rely on stronger sample-level smoothness assumptions and require more expensive gradient updates. To address this gap, we develop a momentum-based stochastic linearized proximal augmented Lagrangian method that requires only one stochastic gradient evaluation per iteration, works under the smoothness condition of the expected objective alone, and admits convergence guarantees directly in terms of the KKT residual of the original problem. The main contributions of this paper are summarized as follows.

\begin{itemize}
    \item We propose a momentum-based stochastic linearized proximal augmented Lagrangian method for solving linearly constrained nonconvex optimization problems. The method integrates a Polyak-type momentum estimator \cite{cutkosky2020momentum} into a single-loop augmented Lagrangian framework and requires only one stochastic gradient evaluation per iteration.
    \item Under the standard stochastic oracle model with unbiased gradients, bounded variance, and smoothness of the expected objective, we establish an $O(\epsilon^{-4})$ complexity bound for computing an $\epsilon$-stationary point. To the best of our knowledge, this is the first complexity guarantees for linearly constrained nonconvex optimization under this weak stochastic setting.
    \item Compared with the recursive-momentum methods that rely on stronger sample-level smoothness assumptions and typically require more expensive gradient updates, the proposed method works under weaker assumptions with a lower per-iteration cost. Simulation results further show that, while achieving iteration complexity comparable to recursive-momentum methods, it attains faster wall-clock convergence, demonstrating its practical efficiency for problems with costly stochastic gradient evaluations.
\end{itemize}

\subsection{Notation}
Unless otherwise specified, $\E[\cdot]$ denotes expectation with respect to all randomness generated by the algorithm. For a vector $x\in\mathbb{R}^d$, $\|x\|$ denotes its Euclidean norm, and for a matrix $A$, $\|A\|$ denotes its induced spectral norm. The notation $A^{\rm T}$ stands for the transpose of $A$, and $I$ denotes the identity matrix of appropriate dimensions. For a differentiable function $f$, $
\nabla f(x)$ denotes its gradient at $x$. We use $\langle \cdot,\cdot
\rangle$ to denote the standard inner product. Throughout the paper, $\{x^r\}$ and $\{\mu^r\}$ denote the primal and dual iterates generated by the algorithm, respectively, and $m^r$ denotes the momentum-based estimator of $
\nabla f(x^r)$. For two symmetric matrices $X$ and $Y$, the notation $X\succeq Y$ means that $X-Y$ is positive semidefinite.

\section{Algorithm Development}\label{sec: Algorithm Development}
We first introduce a standard smoothness assumption on the expected objective.
\begin{assumption}\label{ass: smoothness}
$f: \mathbb{R}^d \to \mathbb{R}$ is lower bounded by $\underline{f}$ and is $L$-smooth, i.e.,
\begin{align*}
    \| \nabla f(x) - \nabla f(y) \| \le L \| x - y \| \quad \forall x, y \in \mathbb{R}^d.
\end{align*}
\end{assumption}

\subsection{Approximated Augmented Lagrangian Function}
Define the augmented Lagrangian function of problem \eqref{form: equality} as follows:
\begin{align}\label{eq: ALF}
    L_{\beta}(x, \mu ) \coloneqq f(x) + \langle \mu, A x- b\rangle + \frac{\beta}{2} \| A x - b \|^2.
\end{align}
Given the current iterate $(x^r,\mu^r)$, we consider the following proximal augmented Lagrangian model
\begin{equation}\label{eq: phi_r}
    \phi_r(x)\coloneqq L_\beta(x,\mu^r)+\frac{\eta}{2}\|x-x^r\|_{}^2.
\end{equation}
Recall that, in the augmented Lagrangian method, the $x$-update at the $(r+1)$th iteration is $x^{r+1}
= 
\arg\min_{x }
\{  \phi_r(x) \}.$
However, directly minimizing $\phi_r(x)$ can be computationally expensive due to the nonlinear objective $f(x)$. 
Instead, we linearize the augmented Lagrangian function $L_{\beta}(x, \mu^r)$ at $x^r$. 
Specifically, the linearized proximal augmented Lagrangian model is given by
\begin{align}
\tilde\phi_r(x)
=& \ 
\langle \nabla f(x^r)+A^{\rm T}\mu^r+\beta A^{\rm T}(A x^r - b),\; x-x^r\rangle \notag\\
& +
\frac{\eta}{2}\|x-x^r\|^2 .
\end{align}
In many data-driven applications, the exact gradient $\nabla f(x^r)$ is unavailable or expensive to compute. Instead, we only have access to stochastic gradients $g(x^r;\xi^r)$. To reduce the variance of stochastic gradients and stabilize the updates, we employ a Polyak-type momentum estimator
\begin{align}\label{upd: momentum}
m^{r+1}=(1-\alpha)m^r+\alpha g(x^{r+1};\xi^{r+1}).
\end{align}
Using $m^r$ as an estimator of $\nabla f(x^r)$, the $x$-update is obtained by minimizing the linearized proximal model
\begin{align*}
x^{r+1}
= &\
\arg\min_{x}
\{
\tilde\phi_r(x)
\}.
\end{align*}
Equivalently, it admits the explicit form
\begin{align}\label{upd: x}
x^{r+1}
=
x^r-\frac{1}{\eta}(m^r+A^{\rm T}\mu^r+\beta A^{\rm T}(A x^r - b)).
\end{align}

The dual variable is then updated by a gradient ascent step on the augmented Lagrangian,
\begin{align}\label{upd: lambda}
\mu^{r+1}=\mu^r+\beta(Ax^{r+1}-b).
\end{align}
The complete procedure is summarized in Algorithm~\ref{alg: main}.
\begin{algorithm}[ht]
\caption{Momentum-based Augmented Lagrangian Method (MALM)}
\label{alg: main}
\begin{algorithmic}[1]
\STATE \textbf{Input:} $x^0\in \mathbb{R}^d$, $\mu^0 = 0$, stepsizes $\eta>0$ and $\beta>0$, and momentum parameter $\alpha\in(0,1)$. Sample $\xi^0 \sim \mathcal{D}$, the initial momentum is $m^0 = g(x^0;\xi^0)$, 
\FOR{$r=0,1,2,\dots$}
\STATE Calculate the primal variable $x^{r+1}$ through \eqref{upd: x}
\STATE Calculate the dual variable $\mu^{r+1}$ through \eqref{upd: lambda}.
\STATE Sample $\xi^r \sim \mathcal{D}$ and calculate the momentum $m^{r+1}$ through \eqref{upd: momentum}. 
\ENDFOR
\end{algorithmic}
\end{algorithm}

\subsection{First-order Condition}
We say that $x^\star\in \mathbb{R}^d$ is a first-order stationary point of problem \eqref{form: equality} if there exists $\mu^\star\in\mathbb{R}^m$ such that
\begin{align}
\nabla f(x^\star)+A^{\rm T}\mu^\star &= 0, \label{eq: stationary-a}\\
\|Ax^\star-b\| &= 0. \label{eq: stationary-b}
\end{align}
Equivalently, $x^\star$ satisfies the first-order KKT conditions of problem \eqref{form: equality}.

We say that $\bar x\in \mathbb{R}^d$ is an $\epsilon$-stationary point of problem \eqref{form: equality} if there exists $\bar\mu\in\mathbb{R}^m$ such that
\begin{align}
\|\nabla f(\bar x)+A^{\rm T}\bar\mu\| &\le \epsilon, \label{eq: eps-stationary-a}\\
\|A\bar x-b\| &\le \epsilon. \label{eq: eps-stationary-b}
\end{align}

\section{Convergence Analysis}
\subsection{Auxiliary Lemmas}
This subsection establishes auxiliary lemmas necessary for our subsequent global convergence and oracle complexity analyses.

\begin{lemma}\label{lem:quad_lower_bound}
Under Assumption \ref{ass: smoothness} and the condition that $\eta \succeq L$, $\phi_r$ is $\gamma$-strongly convex for $\gamma  = \eta - L$.

\end{lemma}

\begin{proof}
By the $L$-smoothness of $f(x)$, we have $\nabla^2 f(x) \succeq - LI$. In view of the penalty terms and the fact $ \beta A^{\rm{T}} A \succeq 0$, the Hessian of $\phi_r(x)$ satisfies
\begin{equation*}
\begin{aligned}
\nabla^2 \phi_r(x) &= \nabla^2 f(x) + \beta A^{\rm{T}} A + \eta I \succeq \eta I-LI  
= (\eta - L)I,
\end{aligned}
\end{equation*}
which yields the strong convexity.
\end{proof}
Since $\phi_r$ is $\gamma$-strongly convex, it satisfies the standard first-order lower bound:
\[
\phi_r(x)\ge \phi_r(z)+\langle \nabla \phi_r(z),x-z\rangle+\frac{\gamma}{2}\|x-z\|^2, \ \  \forall\,x,z\in\mathbb{R}^N.
\]
It then follows from \eqref{eq: phi_r} that
\begin{align}
&L_\beta(x,\mu^r)+\frac{\eta}{2}\|x-x^r\|^2 \nonumber\\
\ge& L_\beta(z,\mu^r)+\frac{\eta}{2}\|z-x^r\|^2 +\frac{\gamma}{2}\|x-z\|^2 \nonumber \\
& +\langle \nabla_x L_\beta(z,\mu^r)+\eta (z-x^r),\,x-z\rangle
. \label{eq: PALF}
\end{align}

\begin{lemma}\label{lem: m^r+1-nabla f}
Suppose Assumptions \ref{ass: variance} and \ref{ass: smoothness} hold. The expected difference between the momentum and the gradient is bounded as: 
\begin{equation}
\begin{aligned}
& \E \left[\left\|m^{r+1}-\nabla f\left(x^{r+1}\right)\right\|^2\right] \\
\le & (1-\alpha)\E \left\|m^r-\nabla f\left(x^r\right)\right\|^2 + \frac{L^2}{\alpha} \E \left\|x^r-x^{r+1}\right\|^2 +  \alpha^2\sigma^2,
\end{aligned}
\end{equation}
where $\alpha \in (0,1)$ and $\sigma$ is defined in Assumption \ref{ass: variance}.
\end{lemma}

\begin{proof}
% We can upper-bound $\E [\|m^{r+1}-\nabla f(x^{r+1})\|^2]$ as follows:
\begin{align*}
& \E \left[\left\|m^{r+1}-\nabla f\left(x^{r+1}\right)\right\|^2\right] \\
= &\E \left[\left\|\E \left[m^{r+1}\right]-\nabla f\left(x^{r+1}\right)\right\|^2\right] \\ & +\E \left[\left\|m^{r+1}-\E \left[m^{r+1}\right]\right\|^2\right] \\
\le & \left\|(1-\alpha)\left(m^r-\nabla f\left(x^{r+1}\right)\right)\right\|^2 \\
& +\alpha^2 \E \left[\left\| g\left(x^{r+1} ; \xi^{r+1}\right) - \nabla f\left(x^{r+1}\right) \right\|^2\right] \\
\le & \left\|(1-\alpha) ( \left(m^r-\nabla f\left(x^r\right)\right)+ \left(\nabla f\left(x^r\right)-\nabla f\left(x^{r+1}\right) ) \right)\right\|^2 \\
& + \alpha^2\sigma^2 \\
\le & (1-\alpha)\left\|m^r-\nabla f\left(x^r\right)\right\|^2 \\
& +\left((1-\alpha)^2 / \alpha\right)\left\|\nabla f\left(x^r\right)-\nabla f\left(x^{r+1}\right)\right\|^2 +  \alpha^2\sigma^2 \\
\le & (1-\alpha)\left\|m^r-\nabla f\left(x^r\right)\right\|^2 + \frac{L^2}{\alpha}\left\|x^r-x^{r+1}\right\|^2 +  \alpha^2\sigma^2,    
\end{align*}  
where the first inequality is due to \eqref{upd: momentum}, the second inequality uses Assumption \ref{ass: variance}, the third inequality uses the convexity of the Euclidean norm, and the fourth inequality is by Assumption \ref{ass: smoothness} and the fact that $(1-\alpha) < 1$.
\end{proof}

\begin{lemma}\label{lem: mr - mr-1}
For any $r>1$, 
\begin{align*}
&\ \E \| m^r - m^{r-1} \|^2 \\
\le &\ \frac{2 \alpha^2}{(1-\alpha)^2} \E \| \nabla f(x^r) - m^r \|^2 + \frac{2 \alpha^2}{(1-\alpha)^2} \sigma^2.
\end{align*}    
\end{lemma}

\begin{proof}
By the calculation of the momentum \eqref{upd: momentum}, for any $r>1$, we have
\begin{align*}
&\ m^r - m^{r-1} \\
=& -\alpha m^{r-1} + \alpha g(x^r;\xi^r)\\  
= &\ ( 1 -\alpha) m^{r-1} + \alpha g(x^r;\xi^r) - m^{r-1}\\
& + \alpha ((\nabla f(x^{r}) - m^{r})- (\nabla f(x^{r}) -m^{r}))\nonumber \\
= &\ \alpha (g(x^{r};\xi^{r}) - \nabla f(x^{r})) + \alpha (m^{r} - m^{r-1}) \\
& + \alpha (\nabla f(x^{r}) - m^{r}),
\end{align*}
which results in 
\begin{align*}
&\ \E \| m^r - m^{r-1} \|^2 \\
= &\ \frac{\alpha^2}{(1-\alpha)^2} \E \| \nabla f(x^r) - m^r + (g(x^{r};\xi^{r}) - \nabla f(x^{r})) \|^2  \\
\le &\ \frac{2 \alpha^2}{(1-\alpha)^2} \E \| \nabla f(x^r) - m^r \|^2 + \frac{2 \alpha^2}{(1-\alpha)^2} \sigma^2,
\end{align*} 
where the second inequality is due to the triangle inequality and Assumption \ref{ass: variance}
\end{proof}

\begin{lemma}\label{lem: nabla f - m}
Given any $\alpha > 0$, we have the following inequality for any $r > 0$
    \begin{align}
        & \langle\nabla f(x^{r+1}) - m^r, x^{r+1} - x^r \rangle \notag \\
        \le &\ \frac{\alpha}{2} \|\nabla f(x^r) - m^r\|^2 + \frac{1}{2 \alpha}\| x^{r+1} - x^r \|^2 \notag \\
        &\ + L \| x^{r+1} - x^r \|^2.
    \end{align}
\end{lemma}

\begin{proof}
For any $r > 0$,
\begin{align*}
    &\ \langle\nabla f(x^{r+1}) - m^r, x^{r+1} - x^r \rangle \\
    \le &\ \langle \nabla f(x^{r+1}) - \nabla f(x^r), x^{r+1} - x^r \rangle \\
    &\ + \langle \nabla f(x^r) - m^r, x^{r+1} - x^r \rangle \\
    \le &\ \frac{\alpha}{2} \|\nabla f(x^r) - m^r\|^2 + \frac{1}{2 \alpha}\| x^{r+1} - x^r \|^2 \\
    &\ + L \| x^{r+1} - x^r \|^2,
\end{align*} 
where the second inequality is by the Young’s inequality and Assumption \ref{ass: smoothness}.
\end{proof}
Define $B \coloneq I - \frac{\beta}{\eta} A^{\rm{T}} A$. The $x$-update \eqref{upd: x} is equivalent to 
\begin{align} \label{upd: equ-x}
x^{r+1}
= &\
\arg\min_{x}
\Big\{
\langle m^r+A^{\rm T}\mu^r,\; x-x^r\rangle + \frac{\beta}{2}\|A x - b\|^2
\notag \\
& \quad \quad \quad \quad \ +\frac{\eta}{2}\|x-x^r\|_B^2
\Big\}.
\end{align}
Then, we have the following lemma for the dual update.
\begin{lemma}\label{lem: Delta mu}
    Suppose Assumptions \ref{ass: variance} and \ref{ass: smoothness} hold and $\{(x^r, \mu^r, m^r)\}$ are generated by Algorithm \ref{alg: main}. Then, we have
    \begin{align}
    & \frac{1}{\beta} \E \|\mu^{r+1}-\mu^r\|^2 \notag \\
    \leq&\ \frac{3 \eta^2 \| B \|^2 }{\beta\underline{\lambda}_A } (\E \|x^{r+1} - x^r\|^2 + \E \|x^r - x^{r-1}\|^2) \notag \\
    & + \frac{6 \alpha ^2}{\beta\underline{\lambda}_A(1-\alpha)^2} \E \| \nabla f(x^r) - m^r \|^2  \notag \\ 
    &+ \frac{6 \alpha^2}{\beta\underline{\lambda}_A(1-\alpha)^2} \sigma^2, \notag 
    \end{align}
where $\underline{\lambda}_A$ is the smallest nonzero eigenvalue of $A^{\rm{T}} A$.
\end{lemma}
\begin{proof}
From the optimality condition of the update \eqref{upd: equ-x}, we have
\begin{align*}
   m^r \!+\! A^{\rm{T}}\mu^{r}\!+\! \beta (A^{\rm{T}}(Ax^{r+1}-b)\!+\!\eta B(x^{r+1}-x^{r}) = 0,
\end{align*}
substitution of which in \eqref{upd: lambda} gives 
\begin{equation}\label{eq: mu{r+1}}
    A^{\rm{T}} \mu^{r+1} = -m^r - \eta B (x^{r+1} - x^r).
\end{equation}
Similarly, for $\mu^r$, we have
\begin{equation}\label{eq: mu{r}}
    A^{\rm{T}} \mu^r = -m^{r-1} - \eta B (x^r - x^{r-1}).
\end{equation}
Combining \eqref{eq: mu{r+1}} and \eqref{eq: mu{r}} yields
\begin{align*}
    &\ A^{\rm{T}} (\mu^{r+1} - \mu^r)\\
    =& -(m^r - m^{r-1}) -\eta B ( (x^{r+1} - x^r) - (x^r - x^{r-1}))
\end{align*}
Taking square of both sides of the above equation gives 
\begin{align*}
    &\ \E \| A^{\rm{T}} (\mu^{r+1} - \mu^r) \|^2 \\
    = &\ \E \|(m^r - m^{r-1}) + \eta B ( (x^{r+1} - x^r) - (x^r - x^{r-1}))\|^2 \\
    \le &\ 3 \E \|m^r - m^{r-1}\|^2 \\
    & + 3 \eta^2 \E \| B \|^2 \|x^{r+1} - x^r\|^2 + 3 \eta^2 \| B \|^2 \E \|x^r - x^{r-1}\|^2 \\
    \le &\ \frac{6 \alpha^2}{(1-\alpha)^2} \E \| \nabla f(x^r) - m^r \|^2 + \frac{6 \alpha^2}{(1-\alpha)^2} \sigma^2 \\
    & + 3 \eta^2 \| B \|^2 \E \|x^{r+1} - x^r\|^2 + 3 \eta^2 \| B \|^2 \E \|x^r - x^{r-1}\|^2,
\end{align*}
where the first inequality is by the triangular inequality and the second inequality is by Lemma \ref{lem: mr - mr-1}. Combining the above inequality with that fact that $\mu^{r+1} - \mu^r \in \text{Range}(A)$
\end{proof}

We derive the difference between two augmented Lagrangian functions at $r$ and $r+1$ by the next lemma. 
\begin{lemma}\label{lem: ALF diff}
Suppose Assumptions \ref{ass: variance} and \ref{ass: smoothness} hold and $\{(x^r, \mu^r, m^r)\}$ are generated by Algorithm \ref{alg: main}. Then we have
\begin{equation*}
\begin{aligned}
 &\ \E [L_{\beta}(x^{r+1},\mu^{r+1})] - \E [L_{\beta}(x^{r},\mu^{r})] \\
 \le &\ -\left( \frac{\eta - L}{2} - L - \frac{1}{2\alpha} - \frac{3 \eta^2 \|B\|^2}{\beta \underline{\lambda}_A} \right ) \E\|x^{r+1}-x^{r}\|^{2} \\
 & \ + \frac{3 \eta^2 \|B\|^2}{\beta \underline{\lambda}_A} \E \|x^{r-1}-x^{r}\|^{2} + \frac{\alpha}{2} \E \|\nabla f(x^r) - m^r\|^2 \\
 &\ + \frac{6 \alpha^2}{\beta\underline{\lambda}_A(1-\alpha)^2} \E \| \nabla f(x^r) - m^r \|^2  \\ 
    &\ + \frac{6 \alpha^2}{\beta\underline{\lambda}_A(1-\alpha)^2} \sigma^2 .
\end{aligned}
\end{equation*} 
\end{lemma}

\begin{proof}
In view of Lemmas ~\ref{lem: nabla f - m} and~\ref{lem: Delta mu}, it follows from \eqref{upd: x} and \eqref{eq: PALF} that
\begin{align*}
 &\ \E [L_{\beta}(x^{r+1},\mu^{r+1})] - \E [L_{\beta}(x^{r},\mu^{r})] \\
 =&\ \E [L_{\beta}(x^{r+1},\mu^{r+1})] - \E [L_{\beta}(x^{r+1},\mu^{r})] -\frac{\eta}{2} \E \|x^{r+1}-x^{r}\|^{2}\\
 & + \E [L_{\beta}(x^{r+1},\mu^{r})] - \E [L_{\beta}(x^{r},\mu^{r})] +\frac{\eta}{2}\E \|x^{r+1}-x^{r}\|^{2} \\
 \le &\ \frac{\E \|\mu^{r+1}-\mu^{r}\|^{2}}{\beta} -\frac{\gamma}{2} \E \|x^{r+1}-x^{r}\|^{2}-\frac{\eta}{2}\E \|x^{r+1}-x^{r}\|^{2}  \\
 &\ +\E [\langle\nabla_{x}L_{\beta}(x^{r+1},\mu^{r})+\eta (x^{r+1}-x^{r}),x^{r+1}-x^{r}\rangle] \\
 = &\ \frac{\E \|\mu^{r+1}-\mu^{r}\|^{2}}{\beta} -\frac{\gamma}{2}\E \|x^{r+1}-x^{r}\|^{2}  \\
 &\ +\E [\langle \nabla_{x}L_{\beta}(x^{r+1},\mu^{r}) - \nabla f(x^{r+1}) + m^r \\
 & \qquad - \beta A^{\rm{T}} (A x^{r+1} - A x^r),x^{r+1}-x^{r}\rangle\\
 &\ +\langle \beta A^{\rm{T}} (A x^{r+1} - A x^r) ,x^{r+1}-x^{r}\rangle ]\\ 
 &\ +\E [ \langle \nabla f(x^{r+1}) - m^r,x^{r+1}-x^{r}\rangle ]\\
 = &\ \frac{\E \|\mu^{r+1}-\mu^{r}\|^{2}}{\beta} -\frac{\gamma}{2} \E \|x^{r+1}-x^{r}\|^{2}-\frac{\eta}{2} \E \|x^{r+1}-x^{r}\|_{B}^{2}  \\
 &\ + \E [\langle \nabla_{x}L_{\beta}(x^{r+1},\mu^{r}) - \nabla f(x^{r+1}) + m^r,x^{r+1}-x^{r}\rangle ]\\
 &\ + \E [\langle \eta B(x^{r+1}-x^{r}),x^{r+1}-x^{r}\rangle]\\ 
 &\ +\E [ \langle  \nabla f(x^{r+1}) - m^r,x^{r+1}-x^{r}\rangle ]\\
 \le &\ \frac{\E \|\mu^{r+1}-\mu^{r}\|^{2}}{\beta} -\frac{\gamma}{2}\E \|x^{r+1}-x^{r}\|^{2}-\frac{\eta}{2}\E \|x^{r+1}-x^{r}\|_{B}^{2} \\
 &\ + \E [\langle  \nabla f(x^{r+1}) - m^r,x^{r+1}-x^{r}\rangle ]\\
 \le &\ -\left( \frac{\eta - L}{2} - L - \frac{1}{2\alpha} - \frac{3 \eta^2 \|B\|^2}{\beta \underline{\lambda}_A} \right ) \E \|x^{r+1}-x^{r}\|^{2} \\
 & \ + \frac{3 \eta^2 \|B\|^2}{\beta \underline{\lambda}_A}\E \|x^{r-1}-x^{r}\|^{2} + \frac{\alpha}{2} \E \|\nabla f(x^r) - m^r\|^2 \\
 &\ + \frac{6 \alpha^2}{\beta\underline{\lambda}_A(1-\alpha)^2}\E \| \nabla f(x^r) - m^r \|^2  \\ 
    &\ + \frac{6 \alpha^2}{\beta\underline{\lambda}_A(1-\alpha)^2} \sigma^2 ,
\end{align*}   
where the first inequality is by \eqref{eq: PALF}, the forth equality is by the definition of the matrix $B$, the second inequality is due to \eqref{upd: x} and the third inequality is by Lemmas \ref{lem: nabla f - m} and \ref{lem: Delta mu}.
\end{proof}
\subsection{Iteration Complexity}
Let \(\Delta x^{r+1}=x^{r+1}-x^r\) denote the step difference, and \(K\) denote the total number of iterations. We next derive the iteration complexity of the proposed method. The following lemma shows upper bounds of \(\sum_{r=0}^{K-1}\|\Delta x^{r+1}\|^2\) and \(\sum_{r=0}^{K}\|\nabla f(x^r)-m^r\|^2\). To derive an explicit complexity bound, we specify the parameter orders as functions of \(K\). The choices of parameters below are made to balance the step difference, the momentum tracking error, and the accumulated stochastic variance in the descent estimate, so that the resulting telescoping argument yields an \(O(K^{-1/2})\) bound on the averaged squared KKT residual.

\begin{lemma} \label{lem: delta x and var}
Suppose the parameters are chosen as
$\alpha = C_{\alpha} K^{-\frac{1}{2}}$, $\eta = C_{\eta} K^{1/2}$, $\theta = C_{\theta} K^{-1/2}$ and $\beta = C_{\beta} K^{1/2}$, where $C_{\alpha} \in (0,1)$, $C_{\theta} > 0$ and $C_\eta>3 L+\frac{2}{C_\alpha}+\frac{2 C_\theta}{C_\alpha^2}$. Define $M:=\frac{C_\eta}{2}-\frac{1}{C_\alpha}-\frac{C_\theta}{C_\alpha^2}-\frac{3L}{2},$ and $C_x:=M-\frac{6C_\eta^2\|B\|^2}{C_\beta\underline{\lambda}_A}
-\frac{6}{C_\beta\underline{\lambda}_A(1-C_\alpha)^2}.$ 
Suppose that
\[
C_\beta>\frac{6}{\underline{\lambda}_AM}
\left(C_\eta^2\|B\|^2+\frac{1}{(1-C_\alpha)^2}\right),
\]
so that \(C_x>0\) and the following estimate holds
\begin{equation*}
\begin{aligned}
& \ C_x K^{\frac{1}{2}} \sum_{r=0}^K \E \|\Delta x^{r+1}\|^{2} + \theta \sum_{r=0}^K \E \| \nabla f(x^r) - m^r \|^2 \\
% &\ + \left(\frac{\alpha}{2} + \frac{8 \alpha^2}{\beta\underline{\lambda}_A(1-\alpha)^2} + \frac{c\alpha}{2} \frac{2 \alpha^2}{(1-\alpha)^2} \right) \| \nabla f(x^r) - m^r \|^2 \\
\le &\ L_{\beta} (x^0, \mu^0) + q \sigma^2 + K\left( \frac{6 \alpha^2}{\beta\underline{\lambda}_A(1-\alpha)^2} + q \alpha^2 \right)\sigma^2, 
\end{aligned}
\end{equation*}
where
\(
p = \frac{\alpha}{2} + \frac{6 \alpha^2}{\beta\underline{\lambda}_A(1-\alpha)^2} \text{ and }
q = \frac{p+\theta}{\alpha}.
\)
Then,
\begin{align}
& \ \frac{1}{K} \sum_{r=0}^K \E \|\Delta x^{r+1}\|^{2}
= O(K^{-\frac{3}{2}}), \label{eq: delta x <} \\
& \ \frac{1}{K} \sum_{r=0}^K \E \| \nabla f(x^r) - m^r \|^2 = O(K^{-\frac{1}{2}}).\label{eq: variance}
\end{align}

\end{lemma}
\begin{proof}
Since
\(
q = \frac{p + \theta}{\alpha},
\)
we have
\(
(1-\alpha)q+p+\theta = q.
\)
Therefore, by combining Lemma \ref{lem: m^r+1-nabla f} and Lemma \ref{lem: ALF diff}, we obtain
\begin{equation*}
\begin{aligned}
&\ \E[ L_{\beta} (x^{r+1}, \mu^{r+1}) ] + q \E \| \nabla f(x^{r+1}) - m^{r+1} \|^2 \\
&\ - \left( \E [L_{\beta} (x^r, \mu^r )] + q \E \| \nabla f(x^r) - m^r \|^2 \right) \\
\le & \ -\left( \frac{\eta - L}{2} - L - \frac{1}{2\alpha}
- \frac{3 \eta^2 \|B\|^2}{\beta \underline{\lambda}_A}
- \frac{q}{\alpha} \right) \E \|\Delta x^{r+1}\|^{2} \\
& \ + \frac{3 \eta^2 \|B\|^2}{\beta \underline{\lambda}_A} \E \|x^{r-1}-x^{r}\|^{2} \\
& \ + \left( \frac{6 \alpha^2}{\beta\underline{\lambda}_A(1-\alpha)^2}
+ q \alpha^2 \right)\sigma^2
- \theta \E \| \nabla f(x^r) - m^r \|^2 .
\end{aligned}
\end{equation*}

Since function $f$ is lower bounded, summing the above inequality from \(r=0\) to \(r=K\), and using the telescoping structure of the
\(\|x^{r-1}-x^r\|^2\) term, \(\|\nabla f(x^0)-m^0\|^2\le \sigma^2\) and $\E [L_{\beta} (x^K, \mu^K )] + q \E \| \nabla f(x^K) - m^K \|^2 > \underline{f}$, yields
% \begin{equation}
% \begin{aligned}
% % &\ L_{\beta} (x^0, \mu^0) + q \|\nabla f(x^0)-m^0\|^2 - \underline{L}_{\beta} \\
% &\ \E L_{\beta} (x^0, \mu^0) + q \E \|\nabla f(x^0)-m^0\|^2 \\
% \ge & \ \left( \frac{\eta - L}{2} - L - \frac{1}{2\alpha}
% - \frac{6 \eta^2 \|B\|^2}{\beta \underline{\lambda}_A} - \frac{q}{\alpha} \right)
% \sum_{r=0}^K \E \|\Delta x^{r+1}\|^{2} \\
% &\ + \theta \sum_{r=0}^K \E \| \nabla f(x^r) - m^r \|^2
% - K\left( \frac{6 \alpha^2}{\beta\underline{\lambda}_A(1-\alpha)^2}
% + q \alpha^2 \right)\sigma^2 .
% \end{aligned}
% \end{equation}
\begin{equation}\label{eq: sum lag refined 2}
\begin{aligned}
&\ \E [L_{\beta} (x^0, \mu^0)] + q \sigma^2 - \underline{f} \\
\ge & \ \left( \frac{\eta - L}{2} - L - \frac{1}{2\alpha}
- \frac{6 \eta^2 \|B\|^2}{\beta \underline{\lambda}_A} - \frac{q}{\alpha} \right)
\sum_{r=0}^K \E \|\Delta x^{r+1}\|^{2} \\
&\ + \theta \sum_{r=0}^K \E \| \nabla f(x^r) \!-\! m^r \|^2\!
-\! K\left(\! \frac{6 \alpha^2}{\beta\underline{\lambda}_A(1-\alpha)^2}
\!+\! q \alpha^2 \!\right)\sigma^2 .
\end{aligned}
\end{equation}

It remains to lower bound the coefficient of
\(\sum_{r=0}^K \|\Delta x^{r+1}\|^2\).
By the definition of \(q\),
\[
\frac{q}{\alpha}
=
\frac{p+\theta}{\alpha^2}
=
\frac{1}{2\alpha}
+\frac{6}{\beta\underline{\lambda}_A(1-\alpha)^2}
+\frac{\theta}{\alpha^2}.
\]
Hence,
\begin{align}
& \frac{\eta - L}{2} - L - \frac{1}{2\alpha}
- \frac{6 \eta^2 \|B\|^2}{\beta \underline{\lambda}_A} - \frac{q}{\alpha} \notag \\
= & \ \frac{\eta}{2} - \frac{3L}{2} - \frac{1}{\alpha}
- \frac{6 \eta^2 \|B\|^2}{\beta \underline{\lambda}_A}
- \frac{6}{\beta\underline{\lambda}_A(1-\alpha)^2}
- \frac{\theta}{\alpha^2}. \label{eq: coeffi}
\end{align}
Substituting
\(
\alpha = C_{\alpha} K^{-1/2}, \)
\(\eta = C_{\eta} K^{1/2}, \)
\( \theta = C_{\theta} K^{-1/2}\) and \(\beta = C_{\beta} K^{1/2},
\)
in \eqref{eq: coeffi},
we have
\begin{align*}
&\ \frac{\eta - L}{2} - L - \frac{1}{2\alpha}
- \frac{6 \eta^2 \|B\|^2}{\beta \underline{\lambda}_A} - \frac{q}{\alpha} \\
= & \ \left(
\frac{C_\eta}{2}
-\frac{1}{C_\alpha}
-\frac{C_\theta}{C_\alpha^2}
-\frac{6 C_\eta^2 \|B\|^2}{C_\beta \underline{\lambda}_A}
\right)K^{1/2} \\
&\ -\frac{3L}{2}
-\frac{6}{C_\beta\underline{\lambda}_A(1-\alpha)^2}K^{-1/2}.
\end{align*}
Since 
\(
\alpha = C_\alpha K^{-1/2}\le C_\alpha <1 \text{ and }
\frac{1}{(1-\alpha)^2}\le \frac{1}{(1-C_\alpha)^2}
\),
\begin{align*}
& \frac{\eta - L}{2} - L - \frac{1}{2\alpha}
- \frac{6 \eta^2 \|B\|^2}{\beta \underline{\lambda}_A} - \frac{q}{\alpha} 
\ge \ C_x K^{1/2},
\end{align*}
where
\[
C_x=M-\frac{6 C_\eta^2\|B\|^2}{C_\beta \underline{\lambda}_A}
-\frac{6}{C_\beta \underline{\lambda}_A(1-C_\alpha)^2}>0.
\]
Substituting this lower bound into \eqref{eq: sum lag refined 2}, we obtain
\begin{align}
& \ C_x K^{1/2} \sum_{r=0}^K \E \|\Delta x^{r+1}\|^{2}
+ \theta \sum_{r=0}^K \E \| \nabla f(x^r) - m^r \|^2 \notag \\
\le & \ L_{\beta} (x^0, \mu^0) + q \sigma^2 - \underline{f}
+ K\left( \frac{6 \alpha^2}{\beta\underline{\lambda}_A(1-\alpha)^2}
+ q \alpha^2 \right)\sigma^2. \label{eq: C_x K1/2 sum x}
\end{align}

Next, we estimate the order of the right-hand side. Since
\(
\alpha = O(K^{-1/2}), 
\beta = O(K^{1/2}) \text{ and }
\theta = O(K^{-1/2}),
\)
we have
\[
p=\frac{\alpha}{2}+\frac{6\alpha^2}{\beta\underline{\lambda}_A(1-\alpha)^2}
=O(K^{-1/2}),
\]
and thus
\(
q=\frac{p+\theta}{\alpha}=O(1).
\)
Moreover,
\[
K\cdot \frac{6 \alpha^2}{\beta\underline{\lambda}_A(1-\alpha)^2}
=O(K^{-1/2}),
\quad
K\cdot q\alpha^2 = O(1).
\]
Hence, the whole right-hand side of \eqref{eq: C_x K1/2 sum x} is of \(O(1)\). Therefore,
\[
\frac1K\sum_{r=0}^K \E \|\Delta x^{r+1}\|^2 = O(K^{-3/2}),
\]
which proves \eqref{eq: delta x <}. 

Similarly, since \(\theta = C_\theta K^{-1/2}\), we have
\[
\frac1K\sum_{r=0}^K \E \|\nabla f(x^r)-m^r\|^2 = O(K^{-1/2})
\]
which proves \eqref{eq: variance}.
\end{proof}

In Lemma \ref{lem: delta x and var}, \eqref{eq: delta x <} implies that the iterates become asymptotically stable in the sense that the successive difference \(\Delta x^{r+1}\) diminishes on average, while \eqref{eq: variance} guarantees that the stochastic gradient estimator \(m^r\) tracks \(\nabla f(x^r)\) with vanishing mean-square error. Building on Lemma \ref{lem: delta x and var}, we next establish the convergence rate of the KKT residual.

\begin{theorem}
Consider the parameter choices in Lemma \ref{lem: delta x and var}. Let \(R\) be uniformly distributed over
\(\{0,1,\ldots,K-1\}\), and define
\(
(\tilde x,\tilde \mu):=(x^{R+1},\mu^{R+1}).
\)
Then, the output pair \((\tilde x,\tilde \mu)\) satisfies
\begin{align}
\E \left[\|\nabla f(\tilde x)+A^{\rm T}\tilde \mu\| + \|A\tilde x-b\|\right]
&= O(K^{-1/4}), \label{eq: rate-thm}
\end{align}
and consequently, for any \(\epsilon>0\), the output pair \((\tilde x,\tilde \mu)\) is an
\(\epsilon\)-stationary solution, i.e.,
\begin{align}
\E \left[\|\nabla f(\tilde x)+A^{\rm T}\tilde \mu\| + \|A\tilde x-b\|\right]
\le \epsilon, \label{eq: eps-stationary-thm}
\end{align}
provided that the total number of iterations satisfies
\begin{align}
K = O(\epsilon^{-4}). \label{eq: complexity-thm}
\end{align}
\end{theorem}

\begin{proof}
We first bound the stationarity residual. From the \(x\)-update \eqref{upd: x},
\begin{equation}
\begin{aligned}
x^{r+1}
&= x^r-\frac{1}{\eta}\!\left(
m^r+A^{\rm{T}}\mu^{r+1}
+\beta A^{\rm{T}} A(x^r-x^{r+1})
\right),
\end{aligned}
\end{equation}
which implies
\begin{equation}
\begin{aligned}
& \ \nabla f(x^{r+1})+A^{\rm{T}}\mu^{r+1} \\
=&\ \eta(x^r-x^{r+1})
+\beta A^{\rm{T}} A(x^{r+1}-x^r)  \\
&\ -(m^r-\nabla f(x^r)) +\nabla f(x^{r+1})-\nabla f(x^r).
\end{aligned}
\end{equation}

Using the triangle inequality and the \(L\)-smoothness of \(f\), we obtain
\begin{align*}
& \ \|\nabla f(x^{r+1})+A^{\rm{T}}\mu^{r+1}\| \notag \\
\le&\
\eta\|x^r-x^{r+1}\|
+  \|m^r-\nabla f(x^r)\| \notag \\
& \  +(\beta\|A\|^2+L)\|x^r-x^{r+1}\|,
\end{align*}
which in view of \eqref{eq: delta x <} and \eqref{eq: variance} yields 
\begin{align}\label{eq: sta_resid}
& \ \frac1K\sum_{r=0}^K \E \|\nabla f(x^{r+1})+A^{\rm{T}}\mu^{r+1}\|^2 = O(K^{-\frac{1}{2}}).
\end{align}

Similarly, we bound the feasibility residual. By the dual update, 
\begin{align*}
    \beta^2 \underline{\lambda}_A \| A x^{r+1} - b \|^2 &=  \underline{\lambda}_A \| \mu^{r+1} - \mu^r \|^2\\
    & \le \| A^{\rm{T}} (\mu^{r+1} - \mu^r) \|^2.
\end{align*}
Averaging both sides and invoking Lemma \ref{lem: Delta mu}, we obtain
\begin{equation}\label{eq: feas_resid}
\frac{1}{K}\sum_{r=0}^{K-1}\E\|Ax^{r+1}-b\|^2
= O(K^{-1/2}).
\end{equation}
Combining \eqref{eq: sta_resid} and \eqref{eq: feas_resid}, and selecting
\(\tilde x = x^{R+1}\) with \(R\) uniformly distributed over
\(\{0,1,\dots,K-1\}\), we obtain \eqref{eq: rate-thm} by Jensen's inequality. Therefore, to ensure
\eqref{eq: eps-stationary-thm}, it suffices to choose \(K\) such that
\(C K^{-1/4}\le \epsilon\), namely,
\begin{align*}
K = O(\epsilon^{-4}).
\end{align*}
This completes the proof.
\end{proof}

\section{Simulation}
To evaluate the proposed method, we consider the following linearly constrained nonconvex regression problem \cite{feng2015learning,gaines2018algorithms}:
\begin{align}
\min_{x }\;
\frac{1}{N}\sum_{i=1}^N \Phi(a_i^{\rm{T}} x-y_i)
\quad \text{s.t.}\, Ax=b,
\end{align}
where
$\Phi(t)=1-\exp\!\left(-t^2\right)$
is the exponential squared loss.
Here, the loss term is smooth but nonconvex. Therefore, this model matches the problem structure studied in this paper.
% In this section, we test the proposed algorithm on a synthetic stochastic linearly constrained problem. The smooth component is chosen as
% \begin{align*}
% f(x)=\frac{1}{N}\sum_{i=1}^N \frac{1}{2}(a_i^{\rm{T}} x-y_i)^2
% +
% \tau \sum_{j=1}^d \sin^2(x_j),
% \end{align*}
% The nonsmooth term is selected as $h(x)=\lambda \|x\|_1$,
% and the constraint set is $X=\{x\in\mathbb{R}^d:\|x\|_\infty \le R\}.$ 
The linear equality constraint is generated synthetically as $Ax=b$, where \(A \in \mathbb{R}^{m\times d}\) is selected randomly from a  Gaussian distribution. To ensure that the constraint set is nonempty, the right-hand side is constructed as $b = A x_{\mathrm{feas}},$ where \(x_{\mathrm{feas}}\) is a randomly generated feasible vector. In this way, the equality constraint introduces linear coupling among decision variables while guaranteeing feasibility.

The dataset is generated synthetically. The vectors $a_i\in\mathbb{R}^d$ are independently sampled from the standard Gaussian distribution, and the observations are generated by $y_i=a_i^{\rm{T}} x^\star+\varepsilon_i,$
where $x^\star$ is a sparse ground-truth vector and $\varepsilon_i$ is Gaussian noise. In the simulation, the problem dimension is set to $d=2000$, the number of samples is $N=5000$, and the number of linear constraints is $m=200$. 
% The sparsity level of $x^\star$ is chosen as $20$. 

We compare the proposed method with two baselines: an augmented Lagrangian method with recursive momentum \cite{alacaoglu2024complexity}, and a stochastic primal-dual method (SPD) \cite{jin2022stochastic}. For the stochastic methods, one stochastic gradient sample is used for every gradient evaluation. For fair comparison, the step sizes and momentum parameters for all algorithms have the same initial states and are carefully tuned to achieve their respective best performances.

To evaluate performance, we report the feasibility violation $\|Ax^r-b\|$, and a stationarity residual, i.e.,
$\| \nabla f(x^r)+A^{\rm{T}} \mu^r \|$. To mitigate the effects of randomness of the generated data, the performance of each algorithm is averaged over 20 independent trials. 
\begin{figure}[!htb]
    \centering     
    \includegraphics[width=0.4\textwidth]{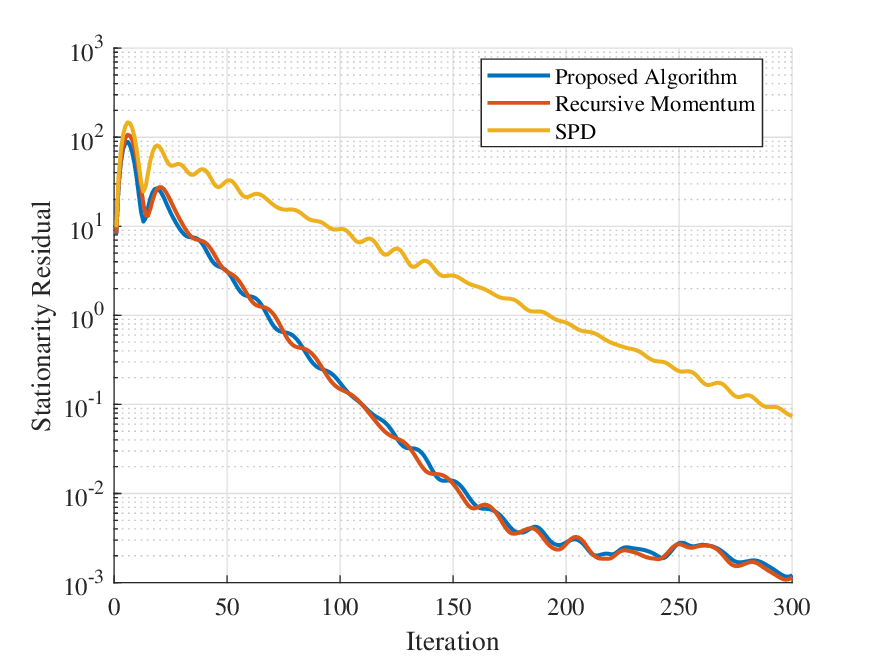}
    \caption{The stationarity residual of primal variables over iteration under the three
algorithms.}
    \label{fig: kkti}
\end{figure}

\begin{figure}[!htb]
    \centering        
    \includegraphics[width=0.4\textwidth]{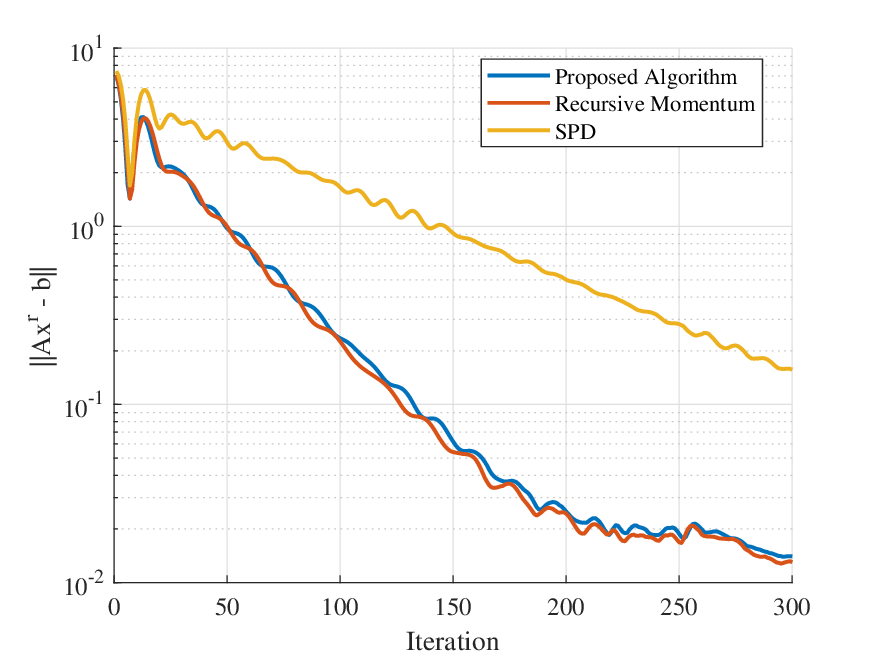}
    \caption{The violation of constraints over iteration under the three algorithms.}
    \label{fig: feasi}
\end{figure}

\begin{figure}[!htb]
    \centering     
    \includegraphics[width=0.4\textwidth]{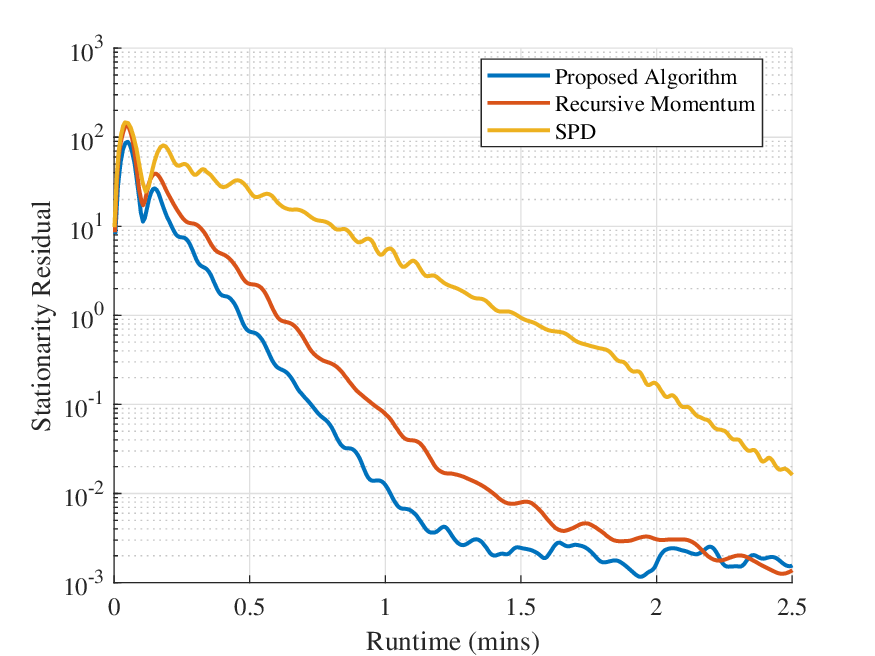}
    \caption{The stationarity residual of primal variables over time under the three
algorithms.}
    \label{fig: kkttime}
\end{figure}

\begin{figure}[!htb]
    \centering        
    \includegraphics[width=0.4\textwidth]{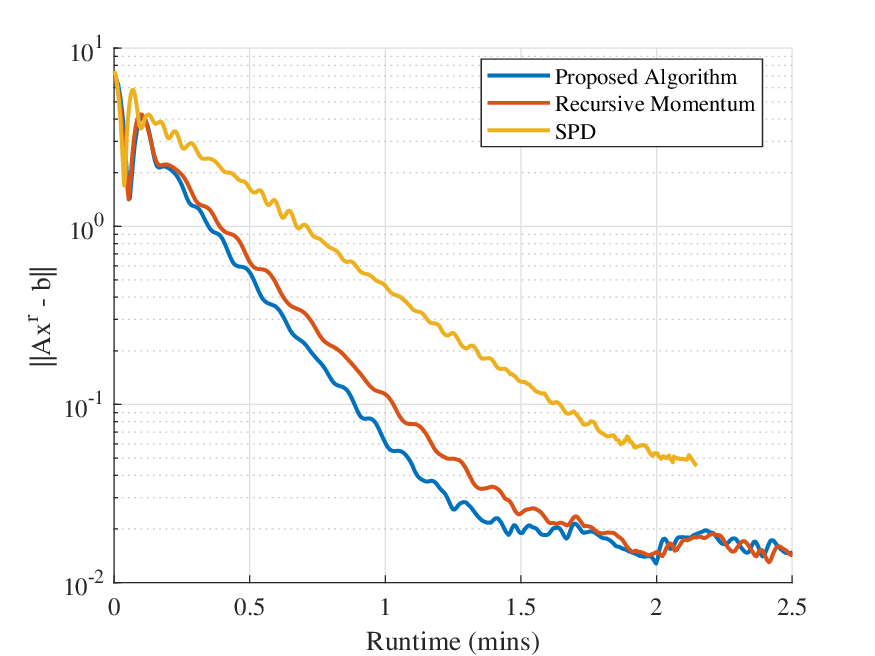}
    \caption{The violation of constraints over time under the three algorithms.}
    \label{fig: feastime}
\end{figure}

\begin{figure}[!htb]
    \centering     
    \includegraphics[width=0.4\textwidth]{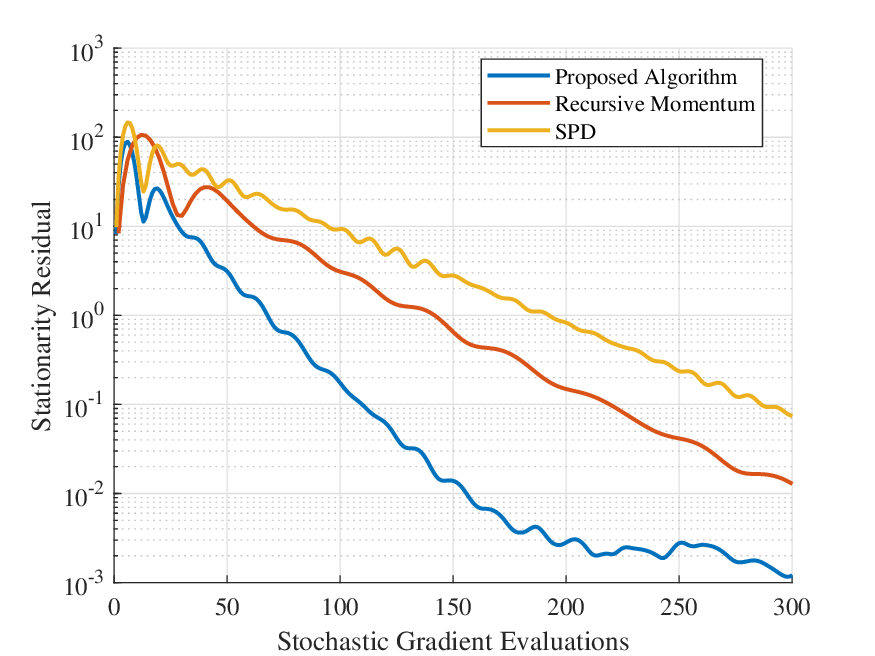}
    \caption{The stationarity residual of primal variables over stochastic gradient evaluations under the three
algorithms.}
    \label{fig: kktgrad}
\end{figure}

\begin{figure}[!htb]
    \centering        
    \includegraphics[width=0.4\textwidth]{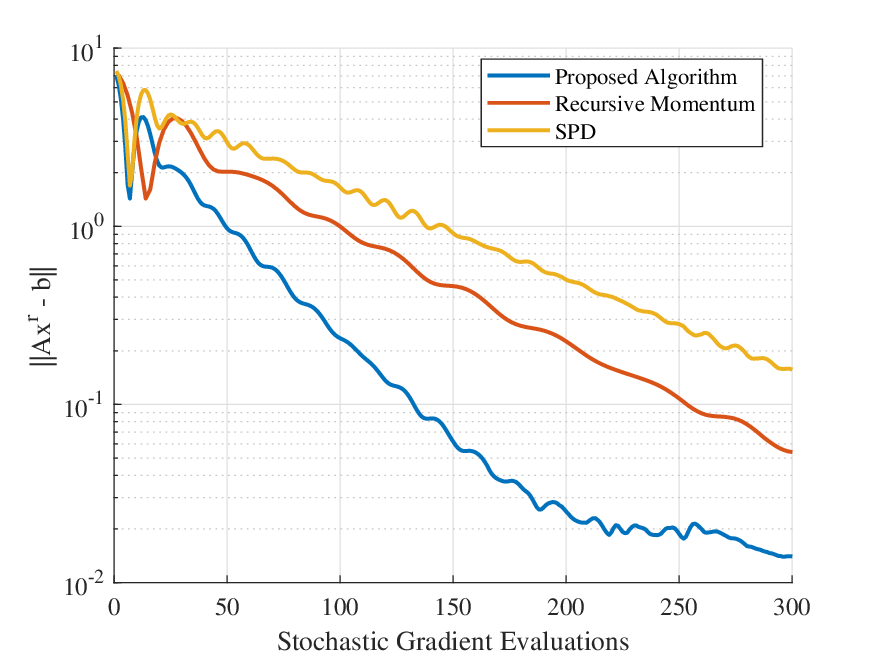}
    \caption{The violation of constraints over stochastic gradient evaluations under the three algorithms.}
    \label{fig: feasgrad}
\end{figure}

Figures \ref{fig: kkti}-\ref{fig: feasgrad} evaluate the algorithms across runtime, stochastic gradient evaluations and iteration counts. As shown in Figures \ref{fig: kkti} and \ref{fig: feasi}, the proposed algorithm and the recursive momentum method exhibit nearly identical iteration complexity, both significantly outperforming the SPD baseline. However, our method's practical advantage emerges in wall-clock time (Figures \ref{fig: kkttime} and \ref{fig: feastime}), where it achieves the fastest convergence. While the recursive momentum method requires computationally heavy updates for variance reduction at each step, our method relies only on a simple Polyak-type momentum and a single stochastic gradient evaluation. In addition, figures \ref{fig: kktgrad} and \ref{fig: feasgrad} plot the stationarity residual and the constraint violation over the number of stochastic gradient evaluations. Since the proposed method and SPD use one stochastic gradient per iteration, while the recursive-momentum baseline requires two stochastic gradient evaluations for the current variable and the history variable per iteration, this horizontal axis provides a fair measure of gradient-estimation cost. 
% As shown in both figures, the proposed algorithm achieves a faster decay of both the stationarity residual and the feasibility error with respect to stochastic gradient evaluations. 
This indicates that the proposed Polyak-type momentum scheme is more efficient in utilizing stochastic gradient information than the recursive-momentum baseline and SPD.
% This balance between competitive iteration complexity and low per-iteration overhead demonstrates its efficiency for computationally demanding applications.

\section{Conclusions}
This paper proposed a momentum-based stochastic linearized proximal augmented Lagrangian method for solving linearly constrained nonconvex optimization problems. By combining a linearized augmented Lagrangian framework with a Polyak-type momentum estimator, the method preserves a simple single-loop structure and requires only one stochastic gradient evaluation per iteration. Under the standard stochastic oracle model, where the stochastic gradients are unbiased with bounded variance and only the expected objective is assumed to be smooth, we established an $O(\epsilon^{-4})$ iteration complexity bound for computing an 
$\epsilon$-stationary solution. Numerical results on a linearly constrained nonconvex sparse regression problem demonstrated that the proposed method has iteration complexity comparable to that of a recursive-momentum augmented Lagrangian baseline, while attaining faster convergence in wall-clock time due to its lower per-iteration cost.

\bibliographystyle{IEEEtran}
\bibliography{reference}
\end{document}